\documentclass{amsart}
\usepackage{amsfonts}
\usepackage{graphicx}
\usepackage{tabularx}
\usepackage{array}
\usepackage[usenames,dvipsnames]{color}
\usepackage{comment}
\usepackage{amsmath}
\usepackage{amsthm}
\usepackage{amssymb}
\usepackage{fullpage}
\usepackage{xcolor}
\usepackage{tikz}

\tikzstyle{legend_general}=[rectangle, rounded corners, thin,
                          top color= white,bottom color=lavander!25,
                          minimum width=2.5cm, minimum height=0.8cm,
                          violet]

\renewcommand{\epsilon}{\varepsilon}

\newtheorem{theorem}{Theorem}[section]

\newtheorem{lemma}[theorem]{Lemma}

\theoremstyle{definition}

\author{Peter Bradshaw}
\address{Simon Fraser University, Burnaby, BC, Canada}
\email{pabradsh@sfu.ca}

\title{Cooperative colorings of forests}

\begin{document}
\maketitle
\begin{abstract}
Given a family $\mathcal G$ of graphs spanning a common vertex set $V$, a \emph{cooperative coloring} of $\mathcal G$ is a collection of one independent set from each graph $G \in \mathcal G$ such that the union of these independent sets equals $V$. We prove that for large $d$, there exists a family $\mathcal G$ of $(1+o(1)) \frac{\log d}{\log \log d}$ forests of maximum degree $d$ that admits no cooperative coloring, which significantly improves a result of Aharoni, Berger, Chudnovsky, Havet, and Jiang (Electronic Journal of Combinatorics, 2020). Our family $\mathcal G$ consists entirely of star forests, and we show that this value for $|\mathcal G|$ is asymptotically best possible in the case that $\mathcal G$ is a family of star forests.
\end{abstract}
\section{Introduction}
In this paper, all graphs and vertex sets that we consider are finite.
Given a family $\mathcal G = \{G_1, \dots, G_k\}$ of graphs that span a common vertex set $V$, a \emph{cooperative coloring} of $\mathcal G$ is defined as a family of sets $R_1, \dots, R_k \subseteq V$ such that for each $1 \leq i \leq k$, $R_i$ is an independent set of $G_i$, and $V = \bigcup _{i = 1}^k R_i$. If each graph $G_i \in \mathcal G$ is equal to a single graph $G$, then the problem of finding a cooperative coloring of $\mathcal G$ is equivalent to the problem of finding a proper $k$-coloring of $G$. Hence, in the cooperative coloring problem, the indices of the graphs in $\mathcal G$ resemble colors used in the traditional graph coloring problem, and we often refer to the indices of graphs in our family $\mathcal G$ as colors.

The cooperative coloring problem is a specific kind of
\emph{independent transversal} problem, which is defined as follows. Given a graph $H$ with a vertex partition $V_1 \cup \dots \cup V_r$, we say that an independent transversal on $H$ is an independent set $I$ in $H$ such that $I$ contains exactly one vertex from each part $V_i$. Given a graph family $\mathcal G = \{G_1, \dots, G_k\}$ on a common vertex set $V$, we can transform the cooperative coloring problem on $\mathcal G$ into an independent transversal problem as follows. We define a graph $H$ with a vertex set $V(H) = V \times [k]$, an edge $(u,i)(v,i)$ for each edge $uv \in E(G_i)$, and with a vertex partition consisting of a part $\{v\} \times [k]$ for each vertex $v \in V$. 
Then, given such a graph $H$ constructed from $\mathcal G$, an independent transversal on $H$ with respect to the parts described above gives a cooperative coloring of $\mathcal G$, and any cooperative coloring of $\mathcal G$ can be transformed into an independent transversal on $H$ after possibly deleting extra vertices to make the independent sets $R_i$ disjoint. 
Certain other graph coloring problems can also be naturally described as independent transversal problems. For example, DP-coloring (also called correspondence coloring) is a recent generalization of list coloring invented by Dvo\v{r}\'ak and Postle \cite{DP}. One way of defining the DP-chromatic number $\chi_{DP}(G)$ of a graph $G$ is with the following statement: $\chi_{DP}(G) \leq k$ if and only if every graph $H$ forming a $k$-sheeted covering space of $G$ with a projection $p:H\rightarrow G$ has an independent transversal with respect to the partition $\bigcup_{v \in V(G)} p^{-1}(v)$ of $V(H)$.

The notion of a cooperative coloring can be naturally generalized to the notion of a \emph{cooperative list coloring}, defined as follows. Consider a graph family $\mathcal G = \{G_1, \dots, G_k\}$ in which each graph $G_i$ has a vertex set $V_i$ that may or may not share vertices with the vertex sets $V_j$ of the other graphs $G_j \in \mathcal G$. We write $V = V_1 \cup \dots \cup V_k$. Then, we say that a \emph{cooperative list coloring} of $\mathcal G$ is a family of vertex subsets $R_1, \dots, R_k$ such that for each value $1 \leq i \leq k$, it holds that $R_i \subseteq V_i$ and $R_i$ is an independent set of $G_i$, and such that $V = \bigcup _{i = 1}^k R_i$. 
Every list coloring problem on a graph $G$ with a list function $L$ can be transformed into a cooperative list coloring problem as follows. For each color $c \in \bigcup_{v \in V(G)} L(v)$, we define the graph $G_c$ to be the subgraph of $G$ induced by those vertices $v \in V(G)$ for which $c \in L(v)$. Then, finding a list coloring on $G$ is equivalent to finding a cooperative list coloring on the family $\mathcal G = \{G_c: c \in \bigcup_{v \in V(G)} L(v) \}$.
The cooperative list coloring problem can also be transformed into an independent transversal problem in a similar way to the cooperative coloring problem. 

One natural question in graph coloring asks for an upper bound on the number of colors needed to color a graph in terms of its maximum degree. Similarly, in the setting of cooperative colorings, we may ask how many graphs of maximum degree $d$ are necessary in a graph family $\mathcal G$ on a common vertex set in order to guarantee the existence of a cooperative coloring. A simple argument using the Lov\'asz Local Lemma shows that if each graph in $\mathcal G$ is of maximum degree at most $d$, then $\mathcal G$ has a cooperative coloring whenever $|\mathcal G| \geq 2ed$. A more involved argument of Haxell \cite{Haxell} implies that $\mathcal G$ is guaranteed a cooperative coloring whenever $|\mathcal G| \geq 2d$. When $d$ is large, Loh and Sudakov \cite{LohSudakov} have shown that a lower bound of the form $|\mathcal G| \geq (1 + o(1))d$ also guarantees the existence of a cooperative coloring on $\mathcal G$. On the other hand, Aharoni, Holzman, Howard, and Spr\"{u}ssel \cite{AHHS} have constructed families containing $d+1$ graphs of maximum degree $d$ spanning a common vertex set that do not admit a cooperative coloring.

For a graph class $\mathcal H$, Aharoni, Berger, Chudnovsky, Havet, and Jiang \cite{CCtrees} defined the parameter $m_{\mathcal H}(d)$ to be the minimum value $m$ for which the following holds: If $\mathcal G$ is a family of at least $m$ graphs of $\mathcal H$ of maximum degree at most $d$ that span a common vertex set, then $\mathcal G$ must have a cooperative coloring. When $\mathcal H$ is the family of all graphs, they write $m(d) = m_{\mathcal H}(d)$. The discussion above implies that $m(d) \leq 2d$ for all values $d \geq 1$, and $m(d) \leq d + o(d)$ asymptotically when $d$ is large. Note that all asymptotics in this paper will be with respect to the parameter $d$, which will always be an upper bound for the maximum degree of each graph in a given graph class.

In a similar fashion to Aharoni, Berger, Chudnovsky, Havet, and Jiang, we will define the parameter $\ell_{\mathcal H}(d)$ for a graph class $\mathcal H$ as follows. 
We say $\ell_{\mathcal H}(d)$ is the minimum value $\ell$ such that if $\mathcal G$ is a family of graphs from $\mathcal H$ of maximum degree at most $d$ whose vertex sets are subsets of a universal vertex set $V$, and if each vertex $v \in V$ belongs to at least $\ell$ graphs in $\mathcal G$, then $\mathcal G$ has a cooperative list coloring. It is straightforward to show that for any graph class $\mathcal H$ and for any value $d$, $m_{\mathcal H}(d) \leq \ell_{\mathcal H}(d)$. When $\mathcal H$ is the class of all graphs, we write $\ell(d) = \ell_{\mathcal H}(d)$. Haxell's argument \cite{Haxell} showing $m(d) \leq 2d$ and Loh and Sudakov's argument \cite{LohSudakov} showing $m(d) \leq d + o(d)$ were both originally formulated for a more general independent transversal problem, and hence their arguments give the same upper bounds on $\ell(d)$ as well. 

We summarize the discussion above with the following inequalities:
\begin{eqnarray}
\label{eqn:haxell}
d + 2 \leq m(d) &\leq& \ell(d) \leq 2d \\
\nonumber
d + 2 \leq m(d) &\leq& \ell(d) \leq d + o(d). 
\end{eqnarray}

In \cite{CCtrees}, Aharoni, Berger, Chudnovsky, Havet, and Jiang considered the value $m_{\mathcal F}(d)$ for the class $\mathcal F$ of forests. 
These authors obtained a lower bound for $m_{\mathcal F}(d)$ from a construction and obtained an upper bound for $m_{\mathcal F}(d)$ by using a creative application of the Lov\'asz Local Lemma that resembles an earlier method used by Bernshteyn, Kostochka, and Zhu \cite[Section 4.2]{DPFrac}, which involves giving each vertex in the problem a random color inventory and then attempting to greedily give each vertex a color from its inventory. Since the method for obtaining an upper bound on $m_{\mathcal F}(d)$ also applies to the cooperative list coloring problem with no changes, we have the following result from \cite{CCtrees}:
 \begin{equation}
\label{eqn:ejc}
\log_2 \log_2 d \leq m_{\mathcal F}(d) \leq \ell_{\mathcal F}(d) \leq (1 + o(1))\log_{4/3} d.
\end{equation}

By using a similar approach involving the Lov\'asz Local Lemma, Bradshaw and Masa\v{r}\'ik \cite{SCC} showed that the upper bound in (\ref{eqn:ejc}) applies not only to forests, but also to graph families of bounded degeneracy $k$ at the expense of a constant factor. In particular, they showed that if $\mathcal H$ is the family of $k$-degenerate graphs, then $\ell_{\mathcal H}(d) \leq 13(1+k\log(kd))$.

In this paper, we will construct a family of forests which we can use to prove that $m_{\mathcal F}(d) \geq (1 + o(1)) \frac{\log d}{\log \log d}$, improving the lower bound in (\ref{eqn:ejc}) significantly.
One interesting feature of our construction is that each graph in our family is a forest of stars. 
Hence, we write $\mathcal S$ for the class of of star forests, and since $\mathcal S \subseteq \mathcal F$, we observe that $m_{\mathcal S}(d) \leq m_{\mathcal F}(d)$.
With $\mathcal S$ defined, we remark that our construction actually implies the stronger lower bound $m_{\mathcal S}(d) \geq (1 + o(1)) \frac{\log d}{\log \log d}$. With a lower bound for $m_{\mathcal S}(d)$ established, it is also natural to ask for an upper bound on $m_{\mathcal S}(d)$. We will prove two results that both imply, as a corollary, that $m_{\mathcal S}(d) \leq \ell_{\mathcal S}(d) \leq  (1+o(1))\frac{\log d}{\log \log d}$, and hence, we will conclude that both $m_{\mathcal S}(d)$ and $\ell_{\mathcal S}(d) $ are of the form $ (1+o(1))\frac{\log d}{\log \log d}$.

\section{A lower bound for $m_{\mathcal S}(d)$}
In this section, we will give a construction that shows that $m_{\mathcal F}(d) \geq m_{\mathcal S}(d) \geq (1+o(1))\frac{\log d}{\log \log d}$. 
For ease of presentation, we will work in the setting of \emph{adapted colorings}, which are defined as follows. Given a multigraph $G$ with a (not necessarily proper) edge coloring $\varphi$, an \emph{adapted coloring} on $(G,\varphi)$ is a (not necessarily proper) vertex coloring $\sigma$ of $G$ in which no edge $e$ is colored the same color as both of its endpoints $u$ and $v$---that is, $\neg \left ( \varphi(e) = \sigma(u) = \sigma(v) \right )$. In other words, if $e \in E(G)$ is a $\mathtt{red}$ edge, then both endpoints of $e$ may not be colored $\mathtt{red}$, but both endpoints of $e$ may be colored, say, $\mathtt{blue}$, and the endpoints of $e$ may also be colored with two different colors.
A cooperative coloring problem on a family $\mathcal G$ may be translated into an adapted coloring problem by coloring the edges of each graph $G_i \in \mathcal G$ with the color $i$ and then considering the multigraph obtained from the union of all graphs in $\mathcal G$, and an adapted coloring problem may be similarly translated into a cooperative coloring problem. Adapted colorings were first considered by Kostochka and Zhu \cite{KostochkaZhu} and have been frequently studied since then \cite{ AdaptablePlanar,HellZhu, MolloyLB, MolloyUB}.

With the adapted coloring framework defined, we are ready to prove our lower bound for $m_{\mathcal S}(d)$. 

\begin{theorem}
\label{thm:LB}
$m_{\mathcal S}(d) \geq (1+o(1))\frac{\log d}{\log \log d}$.
\end{theorem}
\begin{proof}
For each value $t \geq 1$, we will construct a graph $G_t$ whose edges are colored with $\{1, \dots, t\}$ by some function $\varphi_t$ and whose monochromatic subgraphs are star forests. We will show that $(G_t,\varphi_t)$ does not have an adapted coloring with the colors $\{1, \dots, t\}$. Then, we will translate the edge-colored graph $(G_t, \varphi_t)$ into a graph family $\mathcal G_t$ that proves our lower bound. 

We will construct the edge-colored graphs $(G_t, \varphi_t)$ recursively. First, we let $(G_1,\varphi_1)$ be a $K_2$ whose edge is colored with the color $1$. Now, suppose we have constructed $G_t$ along with an edge-coloring $\varphi_t: E(G_t) \rightarrow \{1, \dots, t\}$, and suppose that $(G_t,\varphi_t)$ does not have an adapted coloring with the color set $\{1, \dots, t\}$. For $1 \leq i \leq t+1$, we define a shift function $\psi_i:\{1,\dots,t\} \rightarrow \{1, \dots, t+1\}$ so that 
\[\psi_i(x) = 
\begin{cases}
x & 1 \leq x \leq i-1 \\
x+1 & i \leq x \leq t.
\end{cases}
\]
Now, we construct $(G_{t+1},\varphi_{t+1})$ first by creating $t+1$ disjoint copies $H_1, \dots, H_{t+1}$ of $G_t$, where each $H_i$ is edge-colored with the function $\psi_i \circ \varphi_t$. Observe that $(H_i, \psi_i \circ \varphi_t)$ is isomorphic to $(G_t, \varphi_t)$ as an edge-colored graph, and hence $(H_i, \psi_i \circ \varphi_t)$ does not have an adapted coloring with the colors $\{1, \dots, i-1, i+1, \dots, t+1\}$. Therefore, in any adapted coloring of $(H_i, \psi_i \circ \varphi_t)$ using the color set $\{1, \dots, t+1\}$, some vertex must be colored $i$. Now, we construct $(G_{t+1}, \varphi_{t+1})$ by first taking our $t+1$ disjoint edge-colored copies $(H_i, \psi_i \circ \varphi_t)$ of $G_t$ and adding a single new vertex $v$, and then adding an edge of color $i$ joining $v$ and each vertex of $H_i$, for $1 \leq i \leq t+1$. We call this new graph $G_{t+1}$, and we call its edge-coloring $\varphi_{t+1}$. Observe that by construction, all monochromatic subgraphs of $(G_{t+1}, \varphi_{t+1})$ are star forests. Furthermore, for each value $1 \leq i \leq t+1$, some vertex of $H_i$ must be colored with $i$, and hence no color from the set $\{1, \dots, t+1\}$ is available at $v$. Therefore, $(G_{t+1}, \varphi_{t+1})$ has no adapted coloring using the set $\{1, \dots, t+1\}$.

Now, we compute the maximum degree of each monochromatic subgraph of $G_t$. We write $V_t = |V(G_t)|$, and we write $\Delta_t$ for the maximum number of edges of a single color incident to a vertex in $(G_t, \varphi_t)$. It is easy to see that $\Delta_1 = 1$, $V_1 = 2$, and that the following recursion holds for $t \geq 2$:
\begin{eqnarray*}
\Delta_t &=&V_{t-1}\\
V_t &=& tV_{t-1} + 1.
\end{eqnarray*}
Solving this recurrence, we see that 
\begin{eqnarray*}
V_t &=& V_1 t^{\underline{t-1}} + t^{\underline{t-2}} + \dots +  t^{ \underline{2} }+  t^{ \underline{1} } + 1 = (e+o(1))t! \\
\Delta_t &= & (e+o(1))(t-1)! ,
\end{eqnarray*}
where $t^{\underline{k}} = t! / (t-k)!$ is the falling factorial.

Now, consider a value $d$, and choose $t$ so that $\Delta_t \leq d < \Delta_{t+1}$. We construct $(G_t, \varphi_t)$ as above, and we obtain a graph family $\mathcal G_t = \{G_1, \dots G_t\}$ on the universal vertex set $V(G_t)$ by letting each $G_i \in \mathcal G_t$ have an edge set consisting of those edges of color $i$ in $(G_t, \varphi_t)$. Observe that each graph in $\mathcal G_t$ is a star forest of maximum degree at most $d$. Furthermore, since $(G_t, \varphi_t)$ has no adapted coloring using the color set $\{1, \dots, t\}$, it follows that $\mathcal G_t$ has no cooperative coloring. 
Since $ d \leq (e + o(1))t!$, it follows that $t  \geq (1 + o(1)) \frac{\log d}{\log \log d}$. 
Hence, $m_{\mathcal S}(d) \geq (1 + o(1)) \frac{\log d}{\log \log d}$, completing the proof.
\end{proof}

\section{A partition lemma and an upper bound on $\ell_{\mathcal S}(d)$}
In this section, we aim to show that $\ell_{\mathcal S} (d) \leq (1 + o(1)) \frac{\log d }{\log \log d}$.
In order to prove this upper bound, we establish a partition lemma, which essentially shows that if $\mathcal H$ 
is a graph class whose graphs can be vertex-partitioned into members of classes $\mathcal A$ and $\mathcal B$ for which $\ell_{\mathcal A}(d)$ and $\ell_{\mathcal B}(d)$ are not too large, 
then $\ell_{\mathcal H} (d)$ is also not too large. 
While proving the partition lemma, it is essential that we work in the setting of cooperative list colorings rather than the setting of cooperative colorings. 

While we can use our partition lemma to prove the upper bound $\ell_{\mathcal S} (d) \leq (1 + o(1)) \frac{\log d }{\log \log d}$ directly, we will see that the lemma gives us stronger results that imply this upper bound on $\ell_{\mathcal H}(d)$ as a corollary. We will prove two results that both show an upper bound on $\ell_{\mathcal H}(d)$ for some graph class $\mathcal H$ based on certain forest structures in the graphs of $\mathcal H$, and both of these results will imply that $\ell_{\mathcal S}(d)\leq (1 + o(1)) \frac{\log d }{\log \log d} $.

One tool that we will need to prove the partition lemma is the well-known Lov\'asz Local Lemma, which first appears in a weaker form in \cite{LLL} and can be found in many textbooks, including for example \cite[Chapter 4]{MolloyReed}.

\begin{lemma} \label{lem:LLL}
Let $\mathcal Q$ be a finite set of bad events. Suppose that each event $B \in \mathcal Q$ occurs with probability at most $p$, and suppose further that each event $B \in \mathcal Q$ is dependent with at most $D$ other events $B' \in \mathcal Q$. If 
\[ep(D+1) \leq 1,\]
then with positive probability, no bad event in $\mathcal Q$ occurs.
\end{lemma}

Our partition lemma is as follows.

\begin{lemma}
\label{lem:partition}
Let $\mathcal H$, $\mathcal A$, and $\mathcal B$ be graph classes, and let $ t = t(d)$ be a function of $d$. Suppose that 
\begin{itemize}
\item Each graph $G \in \mathcal H$ of maximum degree at most $d$ can be vertex-partitioned into sets $A$ and $B$ so that $G[A] \in \mathcal A$ and $G[B] \in \mathcal B$, and so that each vertex in $A$ has at most $t$ neighbors in $B$,
\item  $\ell_{\mathcal A}(d) = o(\log d)$,
\item $\ell_{\mathcal B}(d) t = o(\log d)$.
\end{itemize}
Then,
\[\ell_{\mathcal H}(d) \leq (1 + o(1)) \frac{\log d}{\log \log d - \log (\ell_{\mathcal B}(d) t)} + \ell_{\mathcal A}.\]
\end{lemma}

It may help the reader first to visualize $\mathcal A = \mathcal B$ as the class of edgeless graphs and to visualize $\mathcal H = \mathcal S$ as the class of star forests. In this special case, for each star forest $G \in \mathcal H$, we may let $A$ denote the leaf set of $G$ and let $B$ denote the set consisting of the centers of the star components of $G$. In this special case, $\ell_{\mathcal A}(d) = \ell_{\mathcal B}(d) = t = 1$, so the lemma immediately implies that $\ell_{\mathcal S}(d) \leq   (1 + o(1)) \frac{\log d}{\log \log d}$.
\begin{proof}
We fix a value $d$, and 
we consider a family $\mathcal G = \{G_1, \dots, G_k\}$ of graphs from $\mathcal H$ of maximum degree at most $d$ whose vertex sets are subsets of a universal vertex set $V$.
We will write $\ell_{\mathcal A} = \ell_{\mathcal A}(d) $ and $\ell_{\mathcal B} = \ell_{\mathcal B}(d)$.
We assume without loss of generality that each vertex $v \in V$ belongs to exactly $\ell$ graphs in $\mathcal G$. We will show that for each $\gamma > 0$, if $\ell = (1+\gamma)\frac{ \log d}{\log \log d - \log( \ell_{\mathcal B}t )} + \ell_{\mathcal A}$, then when $d$ is sufficiently large, $\mathcal G$ has a cooperative list coloring.

We let $\epsilon > 0$ be a sufficiently small constant (which is at most $1$). For each graph $G_i \in \mathcal G$, we suppose that $V(G_i)$ can be partitioned into sets $A_i$ and $B_i$ satisfying the properties of $A$ and $B$ in the lemma's hypothesis. Note that if every vertex of $V$ belongs to at most $\epsilon (\ell - \ell_{\mathcal A})$ sets $B_i$, then every vertex of $V$ must belong to at least $\ell - \epsilon (\ell - \ell_{\mathcal A}) \geq \ell_{\mathcal A}$ sets $A_i$, and hence a cooperative list coloring on $V$ can be found by taking independent subsets of the graphs $G_i[A_i]$. Therefore, we assume that for some nonempty set $U \subseteq V$ of vertices, each vertex $u \in U$ belongs to more than $\epsilon(\ell - \ell_{\mathcal A})$ sets $B_i$.

Now, for each vertex $u \in U$, we write $\textbf B_u$ for the family of all sets $B_i$ containing $u$, for $1 \leq i \leq k$. Then, we choose a family $\textbf B'_u$ of exactly $\ell_{\mathcal B}$ sets $B_i$ uniformly at random (without replacement) from $\textbf B_u$, and we write $C_u = \{i: B_i  \in \textbf B'_u\}$.
We assign each vertex $u$ a color from $C_u$ so that $\mathcal G[U]$ receives a cooperative list coloring, where $\mathcal G[U] = \{G[U \cap V(G)]:G \in \mathcal G\}$. Note that this is possible, since $|C_u| = \ell_{\mathcal B}$ for each vertex $u \in U$, and since $u \in G_i[B_i]$ for each $i \in C_u$.
After this assignment, if a vertex $v \in V$ has a neighbor $u \in U$ via a graph $G_j$ and $u$ is assigned the color $j$, we then say that $j$ is \emph{unavailable} at $v$. 
If $v \in A_j$ and the color $j$ is not unavailable at $v$, then we say that $j$ is \emph{available} at $v$. Observe that if each uncolored vertex $v \in V$ has at least $\ell_{\mathcal A}$ available colors, then we may extend our cooperative list coloring on $\mathcal G[U]$ to a cooperative list coloring on $\mathcal G$. 
Therefore, for each vertex $v \in V \setminus U$, we define a bad event $X_v$, which is the event that fewer than $\ell_{\mathcal A}$ colors are available at $v$. We will use the Lov\'asz Local Lemma (Lemma \ref{lem:LLL}) to show that with positive probability, no bad event occurs and that we can hence find a cooperative coloring of $\mathcal G$.

Now, consider a vertex $v \in V \setminus U$. Suppose that $v \in  A_j$ for some value $j$. Recall that $v$ has at most $t$ neighbors $u \in B_j$ via $G_j$, and each such neighbor $u$ belonging to $U$ is colored from a randomly chosen set $C_u$ of $\ell_{\mathcal B}$ potential colors. The probability that a given vertex $u \in U \cap N_{G_j}(v)$
 is assigned the color $j$ is at most the probability that $j \in C_u$, which is at most $\frac{\ell_{\mathcal B}}{\epsilon (\ell - \ell_{\mathcal A})}$. Therefore, the probability that $j$ is unavailable at $v$ is at most $\frac{ \ell_{\mathcal B} t}{\epsilon(\ell - \ell_{\mathcal A})}$. Note that this argument remains true even if it is given that some other set of colors has already been made unavailable at $v$. Therefore, since $v$ belongs to at least $\ell - \epsilon(\ell - \ell_{\mathcal A})$ sets $A_i$, $\Pr(X_v)$ is bounded above by the probability that more than 
\[\ell - \epsilon(\ell - \ell_{\mathcal A}) - \ell_{\mathcal A} = (1 - \epsilon)(\ell  - \ell_{\mathcal A})\]
colors are made unavailable at $v$, which is at most 
\[{{\ell}\choose {(1 - \epsilon)(\ell  - \ell_{\mathcal A})}} \left ( \frac{ \ell_{\mathcal B}t}{\epsilon(\ell - \ell_{\mathcal A})} \right )^{(1 - \epsilon)(\ell  - \ell_{\mathcal A})} < 2 ^{\ell} \left ( \frac{\ell_{\mathcal B}t}{\epsilon(\ell - \ell_{\mathcal A})} \right )^{(1 - \epsilon)(\ell  - \ell_{\mathcal A})}  .\]
Since each bad event $X_v$ is dependent with fewer than $\ell t d$ other bad events, the Local Lemma (Lemma \ref{lem:LLL}) tells us that all bad events are avoided with positive probability as long as 
\[ 		2^{\ell} \left ( \frac{ \ell_{\mathcal B} t}{\epsilon(\ell - \ell_{\mathcal A})} \right )^{(1 - \epsilon)(\ell  - \ell_{\mathcal A})}  \cdot \ell t d  \cdot e \leq 1					. \]
Equivalently, by taking the natural logarithm on both sides, no bad event occurs with positive probability as long as 
\[ \ell \log 2 + (1 - \epsilon)(\ell  - \ell_{\mathcal A}) ( \log (\ell_{\mathcal B}t) - \log \epsilon- \log (\ell - \ell_{\mathcal A}))  +\log \ell + \log t + \log d + 1 \leq 0.\]
This inequality can be written more simply as follows:
\[  (1 - \epsilon + o(1))(\ell  - \ell_{\mathcal A}) ( \log (\ell_{\mathcal B}   t)  - \log (\ell - \ell_{\mathcal A}))  +(1+o(1))\log d  \leq 0.\]

We claim that this inequality holds when $d$ is sufficiently large and $\epsilon$ is sufficiently small. Recall that $\ell =  \frac{(1+\gamma) \log d}{\log \log d - \log( \ell_{\mathcal B}t)} + \ell_{\mathcal A}$. When we substitute this value for $\ell$ and assume $d$ is large, we can first write the inequality as 
\[
 (1 - \epsilon + o(1))\left (  \frac{(1 + \gamma) \log d}{\log \log d - \log ( \ell_{\mathcal B}t )}   \right ) ( \log( \ell_{\mathcal B}t) - \log \log d)  + (1+o(1)) \log d    \leq  0 ,
\]
or more simply,
\[
     - (1 - \epsilon + o(1))(1 + \gamma)  \log d  + (1+o(1)) \log d    \leq  0 ,
\]
which holds when $\epsilon$ is sufficiently small and $d$ is sufficiently large. Therefore, with positive probability, our random procedure allows us to complete a cooperative list coloring of $\mathcal G$. Since $\gamma > 0$ can be arbitrarily small, this completes the proof.
\end{proof}

As mentioned before, we can use Lemma \ref{lem:partition} directly to prove that $\ell_{\mathcal S}(d) \leq (1 + o(1)) \frac{\log d}{\log \log d}$,
which shows that the lower bound in Theorem \ref{thm:LB} is best possible up to the $o(1)$ function. We will see that Lemma \ref{lem:partition} also implies much stronger results, and we will prove two such results that both imply this upper bound on $\ell_{\mathcal S}(d)$ as a corollary.








For the first of our results, we will need some definitions. Given a rooted tree $T$ with a root $r$, the \emph{height} of a vertex $v$ in $T$ is the distance from $v$ to $r$, and the height of $T$ is the maximum height achieved over all vertices $v \in V(T)$. Given integers $q \geq 1$ and $h \geq 1$, a \emph{$q$-ary tree of height $h$} is a rooted tree in which every vertex of height at most $h-1$ has exactly $q$ children. Given an integer $k \geq 1$, we write $\log^{(k)} d = \underbrace{ \log \log \dots \log  }_\text{$k$ times} d$. Then, we have the following result.
\begin{theorem}
\label{thm:forbidden}
Let $q \geq 2$ and $h \geq 1$ be fixed integers. If $\mathcal H$ is a family of graphs with no $q$-ary tree of height $h$ as a subgraph, then 
\[\ell_{\mathcal H}(d) \leq (1+o_{q,h}(1)) \frac{\log d}{\log^{(h)} d } + O_q(1).\]
\end{theorem}
\begin{proof}
We will prove the theorem by induction on $h$. 
When $h = 1$, then our hypothesis implies that each graph of $\mathcal H$ has maximum degree $q-1$. Hence, by (\ref{eqn:haxell}), it holds that $\ell_{\mathcal H}(d) \leq 2q - 2$, which is certainly of the form $O_{q}(1)$. Hence, the theorem holds when $h = 1$.

Now, suppose  that $h \geq 2$ and that the graphs of $\mathcal H$ contain no $q$-ary tree of height $h$ as a subgraph. 
We write $k = 2q^h$. We consider a graph $G \in \mathcal H$, and we let $A \subseteq V$ be the set of all vertices $v \in V$ for which $\deg_G(v) < k$. Now, we claim that $G \setminus A$ has no $q$-ary tree subgraph of height $h-1$. 
Indeed, suppose that $G \setminus A$ contains a $q$-ary tree $T$ of height $h-1$ as a subgraph. Since no vertex of $T$ belongs to $A$, this implies that every vertex of $T$ must have degree at least $k$ in $G$. However, 
 since 
\[k = 2q^h >(q^{h-1} -1)q + 2q^{h-1} > (q^{h-1} -1)q + |V(T)| ,\]
we hence can greedily choose a set $N_x$ of $q$ neighbors in $N_G(x) \cap (V(G) \setminus V(T))$ for each of the $q^{h-1}$ leaves $x \in V(T)$ in such a way that the sets $N_x$ are pairwise disjoint. Then, by taking the union of $T$ and the sets $N_x$, we have a $q$-ary tree of height $h$ in $G$, a contradiction. Thus, we conclude that $G \setminus A$ has no $q$-ary tree of height $h-1$.

Now, for each $G$, we define the set $A$ as described above, and we let $B = V(G) \setminus A$. By construction, each vertex of $A$ as at most $k$ neighbors in $B$ via the graph $G$. Furthermore, $G[A]$ belongs to the family $\mathcal A$ of graphs of maximum degree $k$, which satisfies $\ell_{\mathcal A}(d) \leq 2k$ by (\ref{eqn:haxell}), and $G[B]$ belongs to the family $\mathcal B$ of graphs with no subgraph isomorphic to a $q$-ary tree of height $h$. By the induction hypothesis, it holds that $\ell_{\mathcal B}(d) \leq  (1+o_{q,h}(1)) \frac{\log d}{\log^{(h-1)} d } + O_q(1)$. Therefore, we can apply Lemma \ref{lem:partition}.

By applying Lemma \ref{lem:partition} and recalling that $k$ and $\ell_{\mathcal A}(d)$ are constants depending on $q$ and $h$, we conclude that 
\[\ell_{\mathcal H}(d) \leq (1 + o_{q,h}(1)) \frac{\log d}{\log \log d - \log (\ell_{\mathcal B})}.\]
As $h \geq 2$, it holds that $\log \ell_{\mathcal B}(d) = \log \log d - \log^{(h)}d + O_{q,h}(1)$, and thus the theorem is proven.
\end{proof}
To use Theorem \ref{thm:forbidden} to prove that $\ell_{\mathcal S}(d) \leq (1 + o(1)) \frac{\log d}{\log \log d}$, consider a binary tree of height $2$, which has $7$ vertices and $4$ leaves. Since no star forest contains this binary tree as a subgraph, the upper bound on $\ell_{\mathcal S}(d)$ follows from Theorem \ref{thm:forbidden} with $q = h = 2$.

Next, we show that if $\mathcal H$ is a graph class whose graphs have a certain quotient of bounded treedepth, then $\ell_{\mathcal H}(d)$ can be bounded above. For this next theorem, we will need some more definitions. If $G$ is a graph and $U_1, \dots, U_k$ is a partition of $V(G)$, then the \emph{quotient graph} $G / (U_1, \dots, U_k)$ is the graph on $k$ vertices obtained by contracting each part $U_i$ to a single vertex and deleting all resulting loops and parallel edges. 

Given a rooted tree $T$ with a root $r$, we define the \emph{closure} of $T$ as the graph on $V(T)$ in which two vertices $u, v \in V(T)$ are adjacent if and only if $u$ and $v$ form an ancestor-descendant pair. Given a rooted forest $F$, in which each tree component has a root, the closure of $F$ is the union of the closures of the components of $F$. For a graph $G$, if there exists a rooted tree $T$ of height $h-1$ such that $G$ is a subgraph of the closure of $T$, then we say that the \emph{treedepth} of $G$ is at most $h$. The reason for this ``off-by-one error" is that if $T$ has height $h-1$, then the longest path in $T$ with the root as an endpoint contains exactly $h$ vertices.

With these definitions in place, we are ready for our second theorem implying that $\ell_{\mathcal S}(d) \leq (1 + o(1)) \frac{\log d}{\log \log d}$.

\begin{theorem}
\label{thm:td}
Let $0 < \epsilon < \frac{1}{2}$ be a fixed value.
Let $\mathcal H$ be a graph class for which each graph $G \in \mathcal H$ has a partition into parts $U_1, \dots, U_k$ of size at most $t = (\log d)^{\epsilon}$, so that each component of the quotient graph $G / (U_1, \dots, U_k)$ has treedepth at most $h$. Then, $\ell_{\mathcal H}(d) \leq \frac{h - 1 + o_h(1)}{1 - 2 \epsilon} \cdot \frac{\log d}{\log \log d}$.
\end{theorem}
\begin{proof}
We prove the theorem by induction on $h$. When $h = 1$, for each graph $G \in \mathcal H$, the quotient graph $G/(U_1, \dots, U_k)$ is an independent set, so each component of $G$ has at most $t$ vertices. Therefore, $\ell_{\mathcal H}(d) <  2t = o\left (\frac{\log d}{\log \log d} \right )$ by (\ref{eqn:haxell}).

Now, suppose that $h \geq 2$. Consider a graph $G \in \mathcal H$. Let $F$ be a rooted forest subgraph of $G / (U_1, \dots, U_k)$ in which each component has height at most $h-1$ and so that the closure of $F$ contains $G/ (U_1, \dots, U_k)$.
We partition $V(G)$ into parts $A$ and $B$ so that $B$ contains the sets $U_i$ corresponding to the roots of $F$ and $A$ contains all other vertices of $G$. Observe that each component of $G[B]$ contains at most $t$ vertices, and each component $K$ of $G[A]$ can be partitioned using the sets $U_i$ so that the quotient graph of $K$ with respect to this partition has treedepth at most $h-1$. Finally, observe that a vertex $v \in A$ is adjacent to a given vertex $u \in B$ only if $v$ belongs to a set $U_i$, $U_j$ is the root ancestor of $U_i$ in $F$, and $u \in U_j$. Hence, each vertex $v \in A$ has at most $|U_j| \leq t$ neighbors in $B$.

Now, we apply Lemma \ref{lem:partition} to $\mathcal H$.
We let $\mathcal A$ be the graph class defined to satisfy the same conditions of $\mathcal H$ except with $h$ replaced by $h-1$, and we let $\mathcal B$ be the class of graphs whose components each have at most $t$ vertices. By the induction hypothesis, $\ell_{\mathcal A}(d) \leq \frac{(h - 2 + o_h(1))}{1-2\epsilon} \cdot \frac{\log d}{\log \log d}$, and $\ell_{\mathcal B}(d) < 2t$ by (\ref{eqn:haxell}). Since $2t^2 = o(\log d)$, all of the hypotheses Lemma \ref{lem:partition} are satisfied, and we can apply the lemma to $\mathcal H$.

By applying Lemma \ref{lem:partition} and using the induction hypothesis, we see that 
\[\ell_{\mathcal H}(d) \leq (1 + o(1)) \frac{\log d}{\log \log d - 2 \log t} +  \frac{h-2 + o(1)}{1 - 2  \epsilon} \cdot \frac{\log d}{ \log \log d} = \frac{h-1 + o(1)}{1 - 2 \epsilon} \cdot \frac{\log d}{ \log \log d}.\]
 Hence, the theorem is proven.
\end{proof}
In order to use Theorem \ref{thm:td} to prove that $\ell_{\mathcal S}(d) \leq (1 + o(1)) \frac{\log d}{\log \log d}$, we observe that if $G$ is a star forest, then every component of $G$ has treedepth at most $2$, so we can apply Theorem \ref{thm:td} with $ h = 2$ and obtain the upper bound.

\section{Conclusion}
By combining (\ref{eqn:ejc}) and Theorem \ref{thm:LB}, we obtain the following inequality:
\[
(1 + o(1)) \frac{\log d}{\log \log d} \leq m_{\mathcal S}(d) \leq m_{\mathcal F}(d) \leq \ell_{\mathcal F}(d) \leq (1 + o(1))\log_{4/3} d.
\]
While this inequality is certainly much tighter than (\ref{eqn:ejc}), the correct asymptotic growth rates for $m_{\mathcal F}(d)$ and $\ell_{\mathcal F}(d)$ remain open. While we do not have a conjecture for the correct growth rates of these quantities, we remark that if $m_{\mathcal F}(d) = \Theta(\log d)$, then Theorem \ref{thm:forbidden} gives a strong necessary condition for forest families that demonstrate this growth rate. Namely, suppose that $\{\mathcal G_d\}_{d \geq 1}$ is a sequence of forest families such that $|\mathcal G_d| = \Theta (\log d)$, the forests of $\mathcal G_d$ have maximum degree at most $d$, and $\mathcal G_d$ has no cooperative coloring. Then, Theorem \ref{thm:forbidden} implies that for each finite tree $T$, $T$ must appear as a subgraph of infinitely many forests from the families in $\{\mathcal G_d\}_{d \geq 1}$.

\raggedright
\bibliographystyle{plain}
\bibliography{CCbib}

\end{document}